\documentclass[a4paper,11pt,reqno]{amsart}

\usepackage{amssymb,epsf,verbatim}
\usepackage{hyperref,tikz}

\theoremstyle{plain}
\newtheorem{lem}{Lemma}[section]
\newtheorem{prop}[lem]{Proposition}

\newtheorem{thm}[lem]{Theorem}

\theoremstyle{definition}

\theoremstyle{remark}

\newtheorem*{rem}{Remark}
\newtheorem*{eg*}{Example}

\newcommand{\up}[1]{\text{$\uparrow^{#1}$}}
\newcommand{\down}[1]{\text{$\downarrow_{#1}$}}

\newcommand{\onto}{\twoheadrightarrow}

\newcommand{\Hom}{\operatorname{Hom}}
\newcommand{\rInd}{r\text{-}\operatorname{Ind}}
\newcommand{\Ind}{\operatorname{Ind}}
\newcommand{\Res}{\operatorname{Res}}
\newcommand{\rRes}{r\text{-}\operatorname{Res}}
\renewcommand{\t}{\mathbf{t}}
\newcommand{\sym}{\mathfrak{S}}

\begin{document}


\author{Kai Meng Tan}
\author{Wei Hao Teo}
\address{Department of Mathematics, National University of Singapore, Block S17, 10 Lower Kent Ridge Road, Singapore 119076.}
\email[K. M. Tan]{tankm@nus.edu.sg}
\email[W. H. Teo]{tweihao@dso.org.sg}

\title{Sign sequences and decomposition numbers}

\date{May 2011}

\thanks{Support by MOE Academic Research Fund R-146-000-135-112.}

\begin{abstract}
We obtain a closed formula for the $v$-decomposition numbers $d_{\lambda\mu}(v)$ arising from the canonical basis of the Fock space representation of $U_v(\widehat{\mathfrak{sl}}_e)$, where the partition $\lambda$ is obtained from $\mu$ by moving some nodes in its Young diagram, all of which having the same $e$-residue.  We also show that when these $v$-decomposition numbers are evaluated at $v=1$, we obtain the corresponding decomposition numbers for the Schur algebras and symmetric groups.
\end{abstract}

\maketitle

\section{Introduction}

Despite its rich structure, symmetric group algebras in positive characteristic are still not well understood.  Many problems --- some of them fundamental --- remain open.  Among them, one of the most famous is the complete determination of its decomposition numbers.  This has been shown to be equivalent to many other open problems, such as the complete determination of the dimensions of the irreducible modules of symmetric groups, or the complete determination of the  decomposition numbers of the (classical) Schur algebras.

Algebras related to the symmetric group algebras includes the above-mentioned Schur algebras, the $q$-Schur algebras and the Iwahori-Hecke algebras.  Also related is the Fock space representation $\mathcal{F}$ of $U_v(\widehat{\mathfrak{sl}}_e)$.  
The $v$-decomposition numbers $d_{\lambda\mu}(v)$ arising from the canonical basis of $\mathcal{F}$ is shown to give
the corresponding decomposition numbers of the $q$-Schur algebras at complex $e$-th root of unity when evaluated at $v=1$ (see \cite{VV}).  As the decomposition matrix of Schur algebras in characteristic $p$ can be obtained from that of the $q$-Schur algebras at complex $p$-th root of unity by postmultiplying the latter by an adjustment matrix (see \cite{J}), 
$d_{\lambda \mu}(v)|_{v=1}$ may be regarded as a first approximation to the decomposition number $d_{\lambda\mu}$ of the Schur algebra.

While there exist several algorithms to compute the $v$-decomposition numbers $d_{\lambda\mu}(v)$ --- which are parabolic Kazhdan-Lusztig polynomials and are of independent interest --- these are all inherently recursive in nature, and in practice can only be applied to `small' cases.  As such, it is desirable to have closed formulae.  While this may seem too ambitious in general, one can hope for such closed formulae when $\lambda$ and $\mu$ are related in some way.

In the same way, while we are currently far from a complete solution to the determination of decomposition numbers $d_{\lambda\mu}$ of symmetric groups and Schur algebras, one can look for partial solutions when $\lambda$ and $\mu$ are related in some way.  In this regard, Kleshchev \cite{Kleshchev} described $d_{\lambda\mu}$ when $\lambda$ is obtained from $\mu$ by moving one node in its Young diagram.  He introduced the sign sequence induced by the pair $(\lambda,\mu)$ and showed that the decomposition number $d_{\lambda\mu}$ equals the number of latticed subsets for this sign sequence.  His approach was to study the Schur algebra and the symmetric group algebra via the special linear group as an algebraic group, and the Schur functor.  An independent and more elementary combinatorial approach was used in \cite{CMT} to show that the corresponding $v$-decomposition numbers is a sum of monic monomials indexed by the latticed subsets.  

It is natural to ask if one can extend the results of \cite{Kleshchev} and \cite{CMT}, i.e.\ to determine $d_{\lambda\mu}$ and/or $d_{\lambda\mu}(v)$ when $\lambda$ is obtained from $\mu$ by moving $k$ nodes, where $k > 1$. We provide an affirmative answer in this paper for the case where the $k$ nodes moved all have the same residue.  In this case, our main result shows that the $v$-decomposition number $d_{\lambda\mu}(v)$ is a sum of monic monomials indexed by the {\em well-nested latticed paths} for the sign sequence induced by $(\lambda,\mu)$, and that the corresponding decomposition number $d_{\lambda\mu}$ can be obtained by evaluating the $v$-decomposition number $d_{\lambda\mu}(v)$ at $v=1$, and hence equals the number of well-nested latticed paths.  We believe the reader will appreciate the simple definition of (well-nested) latticed paths, as compared to the rather technical one of latticed subsets.

Our approach is combinatorial, similar to that of \cite{CMT}, though we keep the paper self-contained and independent of \cite{CMT}; in fact the results in \cite{CMT} are reproved here.  We obtain the above-mentioned $v$-decomposition numbers as well as the analogue of some relevant branching coefficients for the Fock space in this way.  Assuming Kleshchev's results \cite{Kleshchev} on the branching coefficients of the form $[\Res_{\mathcal{S}_{n-1}} L^{\lambda} : L^{\mu}]$ where $\mu$ is obtained from the partition $\lambda$ of $n$ by removing a normal node, we then show that these $v$-decomposition numbers give the corresponding decomposition numbers for the Schur algebras and the symmetric groups when evaluated at $v=1$.  We also obtain the relevant branching coefficients for the Schur algebras and the symmetric groups.

The paper is organised as follows:  we give a summary of the background theory in the next section.  In Section \ref{S:setup}, we set up the machineries with which we can state our main results.  In Section \ref{S:main}, we state and prove the main results, except for Theorem \ref{T:d} where we postpone its proof to the next section.  In Section \ref{S:proof}, we show that Theorem \ref{T:d} is equivalent to Theorem \ref{T:bij}, a combinatorial statement involving sign sequences and well-nested latticed paths.  We then prove Theorem \ref{T:bij}, thereby completing the proof of Theorem \ref{T:d}.

\section{Preliminaries}

In this section, we give a brief account of the background theory that we shall require.  From now on, we fix an integer $e \geq 2$.  Also, $\mathbb{Z}_e = \mathbb{Z}/e\mathbb{Z}$, and $\mathbb{N}_0$ is the set of non-negative integers.

\subsection{Partitions}
Let $n \in \mathbb{N}_0$.  A partition $\lambda$ of $n$ is a finite weakly decreasing sequence of positive integers summing to $n$.  Denote the set of partitions of $n$ by $\mathcal{P}_n$, and let $\mathcal{P} = \bigcup_{n \in \mathbb{N}_0} \mathcal{P}_n$.  If $\lambda  = (\lambda_1,\dotsc, \lambda_s)$, we write $l(\lambda) =s$.

The Young diagram $[\lambda]$ of $\lambda$ is the set
$$
[\lambda] = \{ (i,j) \mid 1 \leq i \leq l(\lambda), 1 \leq j \leq \lambda_i \}.$$
The elements of $[\lambda]$ are called nodes of $\lambda$.  If $r \in \mathbb{Z}_e$ and $(i,j)$ is a node of $\lambda$ such that $j-i \equiv r \pmod e$, then we say that $(i,j)$ has $e$-residue $r$, and that $(i,j)$ is an $r$-node of $\lambda$. A node $(i,j)$ of $\lambda$ is {\em removable} if $[\lambda] \setminus \{(i,j)\} = [\mu]$ for some partition $\mu$, in which case, we also call $(i,j)$ an {\em indent} node of $\mu$.  A node
$(i_1,j_1)$ is on the {\em right} of another node $(i_2,j_2)$ if $j_1 > j_2$.

Let $t \in \mathbb{Z}$ with $t \geq l(\lambda)$.  The set $$B_t(\lambda) = \{ \lambda_i + t -i \mid 1 \leq i \leq l(\lambda) \} \cup \{ t-i \mid l(\lambda) < i \leq t \}
$$
is the $\beta$-set of $\lambda$ of size $t$.


\subsection{The Jantzen order and a more refined pre-order} \label{S:preorder}

Let $\lambda, \tau \in \mathcal{P}$.  We write $\lambda \to \tau$ if there exists $a,b,i,t \in \mathbb{N}_0$ such that
\begin{itemize}
\item $a > b$,
\item $a \ne b-ie$, and $a, b-ie \in B_t(\lambda)$,
\item $b \ne a-ie$, and $b, a-ie \notin B_t(\lambda)$,
\item $B_t(\tau) = B_t(\lambda) \cup \{b,a-ie\} \setminus \{a,b-ie\}$.
\end{itemize}
The Jantzen order $\geq_J$ on $\mathcal{P}$ is defined
as follows:  $\lambda \geq_J \tau$ if and only if there exist partitions $\mu_0,\dotsc, \mu_s$ ($s \in \mathbb{N}_0$) such that $\mu_0 = \lambda$, $\mu_s = \tau$, and $\mu_{i-1} \to \mu_{i}$ for all $i = 1,\dotsc, s$.  As $\lambda\to \tau$ implies that $\lambda$ dominates $\tau$, the usual dominance order on $\mathcal{P}$ is a refinement of the Jantzen order. ($\lambda$ dominates or is equal to $\tau$ if and only if $\lambda,\tau \in \mathcal{P}_n$ for some $n \in \mathbb{N}_0$, $l(\lambda) \leq l(\tau)$, and $\sum_{i=0}^j \lambda_i \geq \sum_{i=0}^j \tau_i$ for all $j \leq l(\lambda)$.)

Let $r \in \mathbb{Z}_e$, and $t \in \mathbb{N}_0$ with $t \geq l(\lambda)$.  Define $\mathbf{s}_{\lambda,r,t} : \mathbb{N}_0 \to \mathbb{N}_0$ by
$$
\mathbf{s}_{\lambda,r,t}(i) =
\begin{cases}
|B_t(\lambda) \cap \{i\}|, &\text{if } i-t \not\equiv r,r-1 \pmod e; \\
|B_t(\lambda) \cap \{i,i+1\}|, &\text{if } i-t \equiv r-1 \pmod e; \\
|B_t(\lambda) \cap \{i,i-1\}|, &\text{if } i-t \equiv r \pmod e.
\end{cases}
$$
Thus, $\mathbf{s}_{\lambda,r,t}(r+t-1 + je) = \mathbf{s}_{\lambda,r,t}(r+t + je)$ for all admissible $j$.  Also, $\mathbf{s}_{\lambda,r,t}(i) = \mathbf{s}_{\lambda,r,t+1}(i+1)$ for all $i \in \mathbb{N}_0$.

We write $\lambda \sim_r \tau$ if and only if $\mathbf{s}_{\lambda,r,t} = \mathbf{s}_{\tau,r,t}$ for some $t \geq l(\lambda),l(\tau)$ (or equivalently, for {\em all} $t \geq l(\lambda),l(\tau)$).  Then $\sim_r$ is an equivalence relation on $\mathcal{P}$; we write $\t_r^{\lambda}$ for the equivalence class containing $\lambda$.  We define a total order on $\mathcal{P}/\!\!\sim_r$ as follows: $\t_r^{\lambda} > \t_r^{\tau}$ if and only if there exist $t$ and $j$ such that $\mathbf{s}_{\lambda,r,t}(i) = \mathbf{s}_{\tau,r,t}(i)$ for all $i> j$, and $\mathbf{s}_{\lambda,r,t}(j) > \mathbf{s}_{\tau,r,t}(j)$.

\begin{lem} \label{L:Jorder->preorder}
Let $r\in \mathbb{Z}_e$, and let $\lambda,\tau \in \mathcal{P}$.  If $\lambda \geq_J \tau$, then $\t_r^{\lambda} \geq \t_r^{\tau}$.
\end{lem}

\begin{proof}
  It is straightforward to verify that $\lambda \to \mu$ implies $\t_r^\lambda > \t_r^\mu$.
\end{proof}

\subsection{$q$-Schur algebras}

Let $\mathbb{F}$ be an algebraically closed field of characteristic $l$, where $l$ is either zero, or $e$ (if $e$ is prime), or is coprime to $e$.  Let $q \in \mathbb{F}^*$ be such that $e$ is the least integer such that $1+q+\dotsb+q^{e-1}= 0$ (i.e.\ $q$ is a primitive $e$-th root of unity if $l=0$ or $l$ is coprime to $e$, and $q = 1$ if $l=e$).  The $q$-Schur algebra $\mathcal{S}_n = \mathcal{S}_{\mathbb{F},q}(n,n)$ over $\mathbb{F}$ has a distinguished class $\{ \Delta^{\mu} \mid \mu \in \mathcal{P}_n \}$ of modules called Weyl modules.  Each $\Delta^{\mu}$ has a simple head $L^{\mu}$, and the set $\{ L^{\mu} \mid \mu \in \mathcal{P}_n \}$ is a complete set of pairwise non-isomorphic simple modules of $\mathcal{S}_n$.
(Note that $\Delta^{\mu}$ and $L^{\mu}$ are denoted as $W^{\mu'}$ and $F^{\mu'}$ respectively in \cite{DJ}.)  The projective cover $P^\mu$ of $L^\mu$ (or of $\Delta^\mu$) has a filtration in which each factor is isomorphic to a Weyl module; the multiplicity of $\Delta^{\lambda}$ in such a filtration is well-defined, and is equal to the multiplicity of $L^{\mu}$ as a composition factor of $\Delta^\lambda$.  We denote this multiplicity as $d_{\lambda\mu}^l$, which is a decomposition number of $\mathcal{S}_n$.  The Jantzen sum formula (see, for example, \cite{Mathasbook}) provides the following consequence:

\begin{lem} \label{L:decompJorder}
If $d^l_{\lambda\mu} \ne 0$, then $\mu \geq_J \lambda$.
\end{lem}

The decomposition numbers in characteristic $l$ and those in characteristic $0$ are related by an adjustment matrix $A_l$: let $D_l = (d_{\lambda\mu}^l)_{\lambda,\mu \in \mathcal{P}_n}$, then $D_l = D_0 A_l$.  Furthermore, the matrices $A_l$, $D_l$ and $D_0$ are all lower unitriangular with nonnegative entries when the partitions indexing its rows and columns are ordered by a total order extending the dominance order on $\mathcal{P}_n$.  As a consequence, we have

\begin{lem} \label{L:geq}
$d_{\lambda\mu}^l \geq d_{\lambda\mu}^0$.
\end{lem}

The Schur functor maps the Weyl module $\Delta^\mu$ to the Specht module $S^{\mu}$ of the Iwahori-Hecke algebra $\mathcal{H}_n = \mathcal{H}_{\mathbb{F},q}(n)$.  It also maps the simple module $L^\mu$ of $\mathcal{S}_n$ to the simple module $D^{\mu}$ of $\mathcal{H}_n$ if $\mu$ is $e$-regular, and to zero otherwise.  As such, the decomposition matrix of $\mathcal{H}_n$ is a submatrix of $\mathcal{S}_n$.

We note that if $q=1$ (or equivalently, $l=e$), then $\mathcal{S}_n$ is the classical Schur algebra and $\mathcal{H}_n$ is the symmetric group algebra $\mathbb{F}\sym_n$.

\subsection{Restriction and induction}

The $q$-Schur algebra $\mathcal{S}_n$ is isomorphic to a subalgebra of $\mathcal{S}_{n+1}$ so that the restriction functor $\Res_{\mathcal{S}_n}$ and the induction functor $\Ind^{\mathcal{S}_{n+1}}$ make sense.  In fact
$$
\Res_{\mathcal{S}_n} = \bigoplus_{r \in \mathbb{Z}_e} \rRes_{\mathcal{S}_n}, \qquad \Ind^{\mathcal{S}_{n+1}} = \bigoplus_{r \in \mathbb{Z}_e} \rInd^{\mathcal{S}_{n+1}}.
$$
For each $r \in \mathbb{Z}_e$, $\rRes_{\mathcal{S}_n}$ and $\rInd^{\mathcal{S}_{n+1}}$ are exact, preserve direct sums and are adjoint of each other.  Their effects on Weyl modules can be easily described in the Grothendieck group:

\begin{thm} \label{T:Weylbranch}
Let $r \in \mathbb{Z}_e$, $\lambda \in \mathcal{P}_n$ and $\tau \in \mathcal{P}_{n+1}$.  Then
\begin{align*}
[\rRes_{\mathcal{S}_n} \Delta^{\tau}] &= \sum_{\sigma} [\Delta^{\sigma}], \\
[\rInd^{\mathcal{S}_{n+1}} \Delta^{\lambda}] &= \sum_{\mu} [\Delta^{\mu}];
\end{align*}
where the first sum runs over all $\sigma \in \mathcal{P}_{n}$ that can be obtained from $\tau$ by removing a removable $r$-node while the second sum runs over all $\mu \in \mathcal{P}_{n+1}$ that can be obtained from $\lambda$ by adding an indent $r$-node.
\end{thm}

\subsection{The Fock space representation}

The algebra $U_{v}(\widehat{\mathfrak{sl}}_{e})$ is the
associative algebra over $\mathbb{C}(v)$ with generators $e_{r}$,
$f_{r}$, $k_{r}$, $k_{r}^{-1}$ $(0\leq r\leq e-1)$, $d$, $d^{-1}$
subject to some relations (see, for example, \cite[\S4]{L}) which we shall not require.
An important $U_{v}(\widehat{\mathfrak{sl}}_{e})$-module is the Fock space representation
$\mathcal{F}$, which as a $\mathbb{C}(v)$-vector space, has $\mathcal{P}$ as a basis.  For our purposes, an explicit description of the action of $f_r$ on $\mathcal{F}$ will suffice.

Let $\lambda$ be a partition and suppose that the partition $\mu$ is obtained by adding an indent $r$-node $\mathfrak{n}$ to $\lambda$.  Let $N_>(\lambda,\mu)$ be the number of indent $r$-nodes on the right of $\mathfrak{n}$ minus the number of removable $r$-nodes on the right of $\mathfrak{n}$.  We have
\begin{equation} \label{E:f_r}
f_r(\lambda) = \sum_{\mu} v^{N_>(\lambda,\mu)} \mu
\end{equation}
where the sum runs over all partitions $\mu$ that can be obtained from $\lambda$ by adding an indent $r$-node.

In \cite{LT1}, Leclerc and Thibon introduced an involution $x \mapsto
\overline{x}$ on $\mathcal{F}$, having the following properties (among
others):
$$\overline{a(v)x } = a(v^{-1})\overline{x}, \quad
\overline{f_r(x)} = f_r(\overline{x}) \qquad (a(v) \in \mathbb{C}(v),\ x
\in \mathcal{F}).$$

There is a distinguished basis $\{G(\lambda) \mid \lambda \in \mathcal{P}
\}$ 
of $\mathcal{F}$ called the canonical basis, which can be characterised (\cite[Theorem 4.1]{LT1}) as follows:
\begin{itemize}
\item  $G(\lambda )- \lambda \in vL$, where $L$ is the
       $\mathbb{Z}[v]$-lattice in $\mathcal{F}$ generated by $\mathcal{P}$.
\item  $\overline{G(\lambda )}=G(\lambda )$.
\end{itemize}

Let $\left< - , - \right>$ be the inner product on $\mathcal{F}$ with respect to which its basis $\mathcal{P}$ is orthonormal.  Define $d_{\lambda\mu}(v)$ by
$$
d_{\lambda\mu}(v) = \left< G(\mu),\lambda \right>.
$$

We collate together some well-known properties of $d_{\lambda\mu}(v)$.

\begin{thm} \label{T:vdecomp} \hfill
\begin{enumerate}
\item $d_{\mu\mu}(v) = 1$;
\item $d_{\lambda\mu}(v) \in v\mathbb{N}_0[v]$ if $\lambda \ne \mu$;
\item $d_{\lambda\mu}(1) = d^{0}_{\lambda\mu}$;
\item $d_{\lambda\mu}(v) \ne 0$ only if $\lambda \leq_J \mu$;
\item $d_{(\lambda_1,\dotsc, \lambda_r),(\mu_1,\dotsc, \mu_s)}(v) = d_{(\lambda_2,\dotsc, \lambda_r),(\mu_2,\dotsc, \mu_s)}(v)$ if $\lambda_1 = \mu_1$.
\end{enumerate}
\end{thm}

\begin{proof}
Part (i) follows from \cite[7.2]{L}, parts (ii, iii) are proved by Varagnolo and Vasserot \cite{VV}, part (iv) follows from parts (ii, iii) and Lemma \ref{L:decompJorder}, and part (v) is a special case of \cite[Theorem 1(1)]{removal}.
\end{proof}

Theorem \ref{T:vdecomp}(iii) in particular establishes the link between the $v$-decomposition numbers of the Fock space and the decomposition numbers of $q$-Schur algebras in characteristic zero. The canonical basis vector $G(\mu)$ of $\mathcal{F}$ thus corresponds to the projective cover $P^{\mu}$ of $q$-Schur algebras, while the standard basis element $\lambda$ of $\mathcal{F}$ corresponds to the Weyl module $\Delta^{\lambda}$.  Under this correspondence, the action of $f_r$ on $\mathcal{F}$ corresponds to that of $\rInd$ on the $q$-Schur algebras.

\section{Combinatorial setup} \label{S:setup}

In this section, we introduce the notations and set up the combinatorial machineries necessary for this paper.

\subsection{Notations}

Given a subset $A$ of a totally ordered set $(\mathcal{U},\geq)$, and $c,d \in \mathcal{U}$ with $c < d$, write
\begin{align*}
A^{<d} &= \{ a \in A \mid a < d \}; \\
A^{>c} &= \{ a \in A \mid a > c \}; \\
A_{c}^{d} &= A^{>c} \cap A^{<d}.
\end{align*}
We also define $A^{\leq d}$ and $A^{\geq c}$ analogously.

\subsection{Pairings}

Given two disjoint finite subsets $A$ and $B$ of a totally ordered set $(\mathcal{U}, \geq)$, we define its associated path, denoted as $\Gamma(A \text{$\to$} B)$, as follows:  let $A \cup B = \{s_1,s_2,\dotsc, s_k \}$ where $s_1< s_2< \dotsb < s_k$.  For each $i$, represent $s_i$ by the following line segment:
$$
\begin{cases}
\diagup, &\text{if }s_i \in A; \\
\diagdown, &\text{if }s_i \in B.
\end{cases}
$$
Connect these $k$ line segments $s_1,s_2,\dotsc, s_k$ in order, to obtain the path $\Gamma(A \text{$\to$} B)$.

We define a {\em pairing of $A$ to $B$} using $\Gamma(A \text{$\to$} B)$.  An element $x \in A$ is paired with $\pi_{A \to B}(x) \in B$ that is represented by the next $\diagdown$ on the right of and at the `same level' as the $\diagup$ representing $x$, if such a $\diagdown$ exists.  If such a $\diagdown$ does not exist, then $x$ is said to be unpaired.

Denote the subsets consisting of paired and unpaired elements of $A$ by $P_{A \to B}(A)$ and $U_{A \to B}(A)$, and the subsets consisting of paired and unpaired elements of $B$ by $P_{A \to B}(B)$ and $U_{A \to B}(B)$.  Clearly, $\pi_{A\to B} : P_{A\to B}(A) \to P_{A\to B}(B)$ is a bijection.

We write $A \onto B$ if $P_{A\to B}(B) = B$ (equivalently, $U_{A\to B}(B) = \emptyset$).  Note that $A \onto B$ if and only if $\Gamma(A \text{$\to$} B)$ is a Dyck path, i.e.\ a path that has the following property: if its left end is placed at the origin of the $(x,y)$-plane, then the path is contained in the first quadrant $\{ (x,y) \mid x,y \geq 0 \}$.

\begin{eg*}
Let $A = \{ 2,3,10\}$ and $B = \{ 5, 6,8 \}$, with the natural order on $\mathbb{Z}$.  Then $\Gamma(A \text{$\to$} B)$ is
\medskip
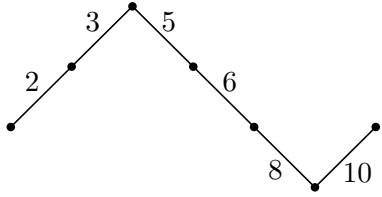
\begin{center}
\begin{tikzpicture}[scale=0.8]
\draw [semithick] (0,0) -- (1,1) -- (2,2) -- (3,1) -- (4,0) -- (5,-1) -- (6,0);
\fill (0,0) circle (2pt);
\fill (1,1) circle (2pt);
\fill (2,2) circle (2pt);
\fill (3,1) circle (2pt);
\fill (4,0) circle (2pt);
\fill (5,-1) circle (2pt);
\fill (6,0) circle (2pt);
\draw (0.35,0.75) node {$2$};
\draw (1.35,1.75) node {$3$};
\draw (2.6,1.75) node {$5$};
\draw (3.6,0.75) node {$6$};
\draw (4.35,-0.7) node {$8$};
\draw (5.7,-0.75) node {$10$};
\end{tikzpicture}
\end{center}
We have $\pi_{A \to B} (2) = 6$, $\pi_{A \to B} (3) = 5$, $U_{A\to B}(A) = \{10\}$ and $U_{A\to B}(B) = \{8\}$.  Also, $A \not\onto B$.
\end{eg*}

If $P_{A\to B}(B) = B$ and $P_{A\to B}(A) = A$, we write $A \leftrightarrow B$.  In this case, $\pi_{A\to B}$ is a bijection from $A$ to $B$.

When $A$ and $B$ are finite subsets of the totally ordered set $(\mathcal{U},\geq)$ but are not necessarily disjoint, let $A^* = A \setminus B$ and $B^* = B \setminus A$.  We extend the pairing $\pi_{A^* \to B^*}$ to $\pi_{A \to B}$ by making every element of $A \cap B$ paired with itself.  The subsets $P_{A\to B}(A)$, $U_{A\to B}(A)$, $P_{A\to B}(B)$ and $U_{A\to B}(B)$ are defined according to the pairing $\pi_{A\to B}$ (thus, for example, $P_{A\to B}(A) = P_{A^* \to B^*}(A^*) \cup (A\cap B)$).  Note then that $A \onto B$ and $A\leftrightarrow B$ if and only if $A^* \onto B^*$ and $A^* \leftrightarrow B^*$ respectively.  Furthermore, $A \onto B$ if and only if $|A^{\leq n}| \geq |B^{\leq n}|$ for all $n \in \mathcal{U}$.

\begin{lem} \label{L:order}
Let $X,Y$ be non-empty disjoint finite subsets of a totally ordered set $(\mathcal{U}, \geq)$.  The relation $\preceq$ on $\{ (A,B) \mid A \subseteq X,\ B \subseteq Y,\ A \onto B \}$, defined by $(A,B) \preceq (C,D)$ if and only if $|A| - |B| = |C| - |D|$ and $(B \cup C) \onto (A \cup D)$, is a partial order.
\end{lem}

\begin{proof}
$\preceq$ is clearly reflexive.

Let $(A,B) \preceq (C,D)$ and $(C,D) \preceq (A,B)$.  To get $(A,B) = (C,D)$, it suffices to show that $A \cup D = B \cup C$, since $A, C \subseteq X$, $B, D \subseteq Y$, and $X \cap Y = \emptyset$.  But $(A \cup D) \onto (B \cup C) \onto (A \cup D)$ gives $|(A\cup D)^{\leq n}| \geq |(B \cup C)^{\leq n}| \geq |(A \cup D)^{\leq n}|$ for all $n \in \mathcal{U}$, so that indeed $A \cup D = B \cup C$.  Thus $\preceq$ is anti-symmetric.

Now let $(A,B) \preceq (C,D)$ and $(C,D) \preceq (E,F)$.  Then $|A|-|B| = |C|-|D| = |E|-|F|$; and
\begin{align*}
|(A \cup D)^{\leq n}| &\leq |(B \cup C)^{\leq n}|, \\
|(C \cup F)^{\leq n}| &\leq |(D \cup E)^{\leq n}|
\end{align*}
for all $n \in \mathcal{U}$.  Since $A, C, E \subseteq X$, $B, D, F \subseteq Y$, and $X \cap Y = \emptyset$, we have \begin{align*}
|A^{\leq n}| + |D^{\leq n}| &\leq |B^{\leq n}| + |C^{\leq n}| ; \\
|C^{\leq n}| + |F^{\leq n}| &\leq |D^{\leq n}| + |E^{\leq n}|.
\end{align*}
Thus, $|(A \cup F)^{\leq n}| = |A^{\leq n}| + |F^{\leq n}| \leq |B^{\leq n}| + |E^{\leq n}| = |(B \cup E)^{\leq n}|$ for all $n \in \mathcal{U}$, i.e.\ $B \cup E \onto A \cup F$.  Hence $(A,B) \preceq (E,F)$, and $\preceq$ is transitive.
\end{proof}


\subsection{Sign sequences}

A sign sequence $T$ is an ordered pair $(T^+,T^-)$ of disjoint finite subsets of a totally ordered set $(\mathcal{U}, \geq)$.  Let $T^{\pm} = T^+ \cup T^-$, and $|T| = |T^+| - |T^-|$.  In addition,
we write $\Gamma(T)$ for $\Gamma(T^+\text{$\to$} T^-)$, and
\begin{align*}
P^+(T) = P_{T^+ \to T^-}(T^+), \qquad U^+(T) = U_{T^+ \to T^-}(T^+); \\
P^-(T) = P_{T^+ \to T^-}(T^-), \qquad U^-(T) = U_{T^+ \to T^-}(T^-).
\end{align*}
Furthermore, let
$$V(T) = \{ v \in T^- \mid (T^+)^{>v} \onto (T^-)^{>v} \}.$$
Thus each element of $V(T)$ indexes the $\diagdown$ at a `valley' in $\Gamma(T)$ where the subpath to its right is a Dyck path.

The {\em latticed paths for $T$} are defined as follows: $\Gamma(T)$ is the generic latticed path for $T$, and all other latticed paths for $T$ are connected paths obtained by replacing one or more `ridges' in the generic latticed path by horizontal line segment(s).  Denote the set of latticed paths for $T$ by $\mathbb{L}(T)$, and for each $\gamma \in \mathbb{L}(T)$, define its norm $\| \gamma \|$ as one plus the total number of $\diagup$ and $\diagdown$ in it.

\begin{eg*}
Suppose that $\Gamma(T)$ is
\medskip
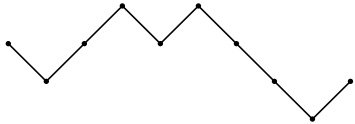
\begin{center}
\begin{tikzpicture}[scale=0.5]
\draw [semithick] (-1,1) -- (0,0) -- (1,1) -- (2,2) -- (3,1) -- (4,2) -- (5,1) -- (6,0) -- (7,-1) -- (8,0);
\fill (-1,1) circle (2pt);
\fill (0,0) circle (2pt);
\fill (1,1) circle (2pt);
\fill (2,2) circle (2pt);
\fill (3,1) circle (2pt);
\fill (4,2) circle (2pt);
\fill (5,1) circle (2pt);
\fill (6,0) circle (2pt);
\fill (7,-1) circle (2pt);
\fill (8,0) circle (2pt);
\end{tikzpicture}
\end{center}
The following are all the non-generic latticed paths for $T$:
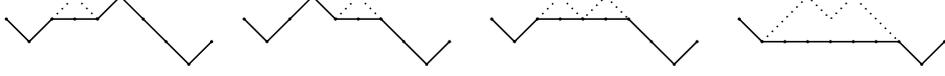
\medskip
\begin{center}
\begin{tikzpicture}[scale=0.3]
\draw [semithick] (-1,1)-- (0,0) -- (1,1) -- (2,1) -- (3,1) -- (4,2) -- (5,1) -- (6,0) -- (7,-1) -- (8,0);
\fill (-1,1) circle (2pt);
\fill (0,0) circle (2pt);
\fill (1,1) circle (2pt);
\fill (2,1) circle (2pt);
\fill (3,1) circle (2pt);
\fill (4,2) circle (2pt);
\fill (5,1) circle (2pt);
\fill (6,0) circle (2pt);
\fill (7,-1) circle (2pt);
\fill (8,0) circle (2pt);
\draw [semithick,dotted] (1,1) -- (2,2) -- (3,1);
\end{tikzpicture}\quad
\begin{tikzpicture}[scale=0.3]
\draw [semithick] (-1,1)-- (0,0) -- (1,1) -- (2,2) -- (3,1) -- (4,1) -- (5,1) -- (6,0) -- (7,-1) -- (8,0);
\fill (-1,1) circle (2pt);
\fill (0,0) circle (2pt);
\fill (1,1) circle (2pt);
\fill (2,2) circle (2pt);
\fill (3,1) circle (2pt);
\fill (4,1) circle (2pt);
\fill (5,1) circle (2pt);
\fill (6,0) circle (2pt);
\fill (7,-1) circle (2pt);
\fill (8,0) circle (2pt);
\draw [semithick,dotted] (3,1) -- (4,2) -- (5,1);
\end{tikzpicture} \quad
\begin{tikzpicture}[scale=0.3]
\draw [semithick] (-1,1)-- (0,0) -- (1,1) -- (2,1) -- (3,1) -- (4,1) -- (5,1) -- (6,0) -- (7,-1) -- (8,0);
\fill (-1,1) circle (2pt);
\fill (0,0) circle (2pt);
\fill (1,1) circle (2pt);
\fill (2,1) circle (2pt);
\fill (3,1) circle (2pt);
\fill (4,1) circle (2pt);
\fill (5,1) circle (2pt);
\fill (6,0) circle (2pt);
\fill (7,-1) circle (2pt);
\fill (8,0) circle (2pt);
\draw [semithick,dotted] (1,1) -- (2,2) -- (3,1) -- (4,2) -- (5,1);
\end{tikzpicture}
\quad
\begin{tikzpicture}[scale=0.3]
\draw [semithick] (-1,1)-- (0,0) -- (1,0) -- (2,0) -- (3,0) -- (4,0) -- (5,0) -- (6,0) -- (7,-1) -- (8,0);
\fill (-1,1) circle (2pt);
\fill (0,0) circle (2pt);
\fill (1,0) circle (2pt);
\fill (2,0) circle (2pt);
\fill (3,0) circle (2pt);
\fill (4,0) circle (2pt);
\fill (5,0) circle (2pt);
\fill (6,0) circle (2pt);
\fill (7,-1) circle (2pt);
\fill (8,0) circle (2pt);
\draw [semithick,dotted] (0,0) -- (1,1) -- (2,2) -- (3,1) -- (4,2) -- (5,1) -- (6,0);
\end{tikzpicture}\end{center}
The norm of the generic latticed path is 10, while the norms of the non-generic ones are 8, 8, 6, 4 respectively.
\end{eg*}

For $a,b \in T^{\pm}$ with $a < b$, we define the following sign subsequences of $T$:
\begin{align*}
T_a^b &= ((T^+)^b_a, (T^-)_a^b) ; \\
T_a &= ((T^+)^{>a}, (T^-)^{>a}); \\
T^b &= ((T^+)^{<b}, (T^-)^{<b}).
\end{align*}

For convenience, we define the empty path to be the only latticed path for $T_a^a$, with zero norm.

Let $A, B \subseteq T^{\pm}$ such that $A \leftrightarrow B$, and let $\pi = \pi_{A\to B}$.  A {\em well-nested latticed path for $(T,A,B)$} is a collection $(\gamma_a)_{a\in A}$ such that for each $a \in A$, $\gamma_a \in \mathbb{L}(T_a^{\pi(a)})$, and no part of $\gamma_a$ falls below the subpath of $\gamma_{a'}$ between $a$ and $\pi(a)$ whenever $a' < a < \pi(a) < \pi(a')$.  We denote the set of well-nested latticed paths for $(T,A,B)$ by $\Omega(T_A^B)$, and for each $\omega = (\gamma_a)_{a\in A} \in \Omega(T_A^B)$, we define its norm $\|\omega\|$ as $\sum_{a\in A} \|\gamma_a\|$.


\begin{rem}
Our definition of sign sequence generalises the original one used by Kleshchev in \cite{Kleshchev} (and subsequently used in \cite{CMT}).  His definition \cite[Definition 1.1]{Kleshchev} corresponds naturally to our sign sequences $T$ with $T^{\pm} = \{1,2,\dotsc, |T^{\pm}|\}$, naturally ordered.  For such a sign sequence $T$,  and $\gamma \in \mathbb{L}(T)$, let $$X_{\gamma} = \{ r \in T^{\pm} \mid r \text{ indexes $\diagdown$ in $\gamma \}$}.$$ Then $X_{\gamma}$ is a latticed subset for $T$ in the sense of \cite[Definition 1.8]{Kleshchev}.  Furthermore, $\| \gamma \| = 1 + 2|X_{\gamma}| + |T|$.  Conversely, given a latticed subset $X$ for $T$, one can find a unique $\gamma \in \mathbb{L}(T)$ such that $X_\gamma = X$.  There is thus a one-to-one correspondence between latticed paths and latticed subsets.  We leave the proof of these facts as a combinatorial exercise for the reader.
\end{rem}


\section{Main results} \label{S:main}

We state and prove the main results in this section, with the exception of Theorem \ref{T:d}, whose proof is postponed to the next section.

Let $\lambda \in \mathcal{P}$, and let $r \in \mathbb{Z}_e$.  Denote the sets of removable $r$-nodes and indent $r$-nodes of $\lambda$ by $R_r(\lambda)$ and $I_r(\lambda)$ respectively.  Let $T_r(\lambda)$ be the sign sequence $(R_r(\lambda), I_r(\lambda))$, where $R_r(\lambda) \cup I_r(\lambda)$ is totally ordered as follows: $a > b$ if and only if $a$ is on the right of $b$.

For $X \subseteq I_r(\lambda)$ and $Y \subseteq R_r(\lambda)$, we write $\lambda \up{X} \down{Y}$ for the partition obtained from $\lambda$ by adding all the indent $r$-nodes $x$ with $x \in X$ and removing all the removable $r$-nodes $y$ with $y \in Y$.  We extend this notation to all subsets $X, Y \subseteq R_r(\lambda) \cup I_r(\lambda)$ such that $X\setminus Y \subseteq I_r(\lambda)$ and $Y \setminus X \subseteq R_r(\lambda)$, i.e.\ for these subsets, $\lambda \up{X} \down{Y} = \lambda \up{X \setminus Y} \down{Y \setminus X}$.

We note that $\t_r^\lambda = \t_r^{\mu}$ (see Section \ref{S:preorder}) if and only if $\mu = \lambda \up{X} \down{Y}$ for some $X \subseteq I_r(\lambda)$ and $Y \subseteq R_r(\lambda)$.

\begin{lem} \label{L:dom}
Let $\lambda \in \mathcal{P}_n$ and $r \in \mathbb{Z}_e$. Let $A \subseteq I_r(\lambda)$ and $B \subseteq R_r(\lambda)$ with $|A| = |B|$.  If $d_{\lambda \up{A} \down{B},\lambda}(v) \ne 0$ or $d_{\lambda \up{A}\down{B}, \lambda}^l \ne 0$, then $A \onto B$.
\end{lem}

\begin{proof}
If $d_{\lambda \up{A}\down{B}, \lambda}(v) \ne 0$ or $d_{\lambda \up{A}\down{B}, \lambda}^l \ne 0$, then $\lambda \geq_J \lambda \up{A}\down{B}$ by Theorem \ref{T:vdecomp}(iv) and Lemma \ref{L:decompJorder}. In particular, $\lambda$ dominates $\lambda \up{A}\down{B}$, so that $A \onto B$.
\end{proof}

\begin{prop} \label{P:impt}
Let $\lambda \in \mathcal{P}_n$ and $r \in \mathbb{Z}_e$. Let $f_r (G(\lambda)) = \sum_{\tau \in \mathcal{P}_{n+1}} a_{\tau\lambda}(v) G(\tau)$.
\begin{enumerate}
\item If $\sigma \in \mathcal{P}_{n+1}$ such that $\langle f_r(G(\lambda)), \sigma \rangle \ne 0$, then $\t_r^{\sigma} \leq \t_r^{\lambda}$.
\item If $\sigma \in \mathcal{P}_{n+1}$ such that $a_{\sigma\lambda}(v) \ne 0$, then $\t_r^{\sigma} \leq \t_r^{\lambda}$.
\item If $A \subseteq I_r(\lambda)$ and $B \subseteq R_r(\lambda)$ with $|A| = |B| + 1$, then
$$
\langle f_r (G(\lambda)), \lambda \up{A} \down{B} \rangle =
\begin{cases}
\sum_{(C,D) \preceq (A,B)} a_{\lambda \up{C} \down{D},\lambda}(v) d_{\lambda \up {A} \down{B},\lambda \up{C}\down{D}}(v), &\text{if } A \onto B; \\
0, &\text{otherwise.}
\end{cases}
$$
\end{enumerate}
\end{prop}

\begin{proof} \hfill
\begin{enumerate}
\item We have $f_r(G(\lambda)) = \sum_{\mu \in \mathcal{P}_n} d_{\mu\lambda}(v) f_r(\mu)$, so that if $\langle f_r(G(\lambda)), \sigma \rangle \ne 0$, then $\sigma = \mu \up{\{c\}}$ for some $\mu \in \mathcal{P}_n$ (with $d_{\mu\lambda}(v) \ne 0$) and $c \in I_r(\mu)$ by \eqref{E:f_r}.  Thus
    $$
    \t_r^\sigma = \t_r^{\mu} \leq \t_r^\lambda
    $$
    by Lemmas \ref{L:decompJorder} and \ref{L:Jorder->preorder}.

\item We have 
\begin{equation} \label{E:f_rG}
\langle f_r(G(\lambda)), \sigma \rangle = \langle \sum_{\tau \in \mathcal{P}_{n+1}} a_{\tau\lambda}(v) G(\tau), \sigma\rangle =  \sum_{\tau \in \mathcal{P}_{n+1}} a_{\tau\lambda}(v) d_{\sigma\tau}(v).
\end{equation}
Thus if $a_{\sigma\lambda}(v) \ne 0$, then
\begin{equation} \label{E:f_rG2}
\langle f_r(G(\lambda)), \sigma \rangle = a_{\sigma \lambda}(v) + \sum_{\tau \ne \sigma} a_{\tau \lambda}(v) d_{\sigma \tau}(v) \ne 0
\end{equation}
since $a_{\tau \lambda}(v) \in \mathbb{N}_0[v,v^{-1}]$ by \cite[Proposition 2.4]{CMT} and $d_{\sigma \tau}(v) \in v\mathbb{N}_0[v]$ by Theorem \ref{T:vdecomp}(ii) for all $\tau$.  Thus, $\t_r^{\sigma} \leq \t_r^{\lambda}$ by part (i).
\item  We note first that if $X \subseteq I_r(\lambda)$ and $Y \subseteq R_r(\lambda)$ with $|X| = |Y|+1$ such that $\langle f_r (G(\lambda)), \lambda \up{X} \down{Y} \rangle \ne 0$, then $X \onto Y$.  This is because
    $$
    0 \ne \langle f_r (G(\lambda)), \lambda \up{X} \down{Y} \rangle = \langle \sum_{\mu} d_{\mu\lambda}(v) f_r(\mu), \lambda \up{X}\down{Y} \rangle,$$
so that $d_{\lambda \up{X}\down{Y \cup\{z\}}, \lambda}(v) \ne 0$ for some $z \in R_r(\lambda) \cup X \setminus Y$ by \eqref{E:f_r}.  Thus $X \onto Y \cup \{z\}$ by Lemma \ref{L:dom}, and hence $X \onto Y$.

Now, assume that $\langle f_r (G(\lambda)), \lambda \up{A} \down{B} \rangle \ne 0$.  Then $A \onto B$.
For each $\sigma \in \mathcal{P}_{n+1}$ such that $a_{\sigma\lambda}(v) d_{\lambda \up{A} \down{B}, \sigma}(v) \ne 0$, we have, clearly, $a_{\sigma\lambda}(v),\ d_{\lambda \up{A} \down{B}, \sigma}(v) \ne 0$.  By part (ii) and Lemmas \ref{T:vdecomp}(iv) and \ref{L:Jorder->preorder}, this gives $$
\t_r^{\lambda} \geq \t_r^{\sigma} \geq \t_r^{\lambda \up {A}\down{B}} = \t_r^{\lambda},$$
so that we must have equality throughout.  Thus $\sigma = \lambda \up {C} \down {D}$ for some $C \subseteq I_r(\lambda)$ and $D \subseteq R_r(\lambda)$ with $|C|= |D|+1$.  Furthermore, since $a_{\sigma\lambda}(v) \ne 0$, we also have
$$\langle f_r (G(\lambda)), \lambda \up{C}\down{D} \rangle = \langle f_r (G(\lambda)), \sigma \rangle \ne 0$$ by \eqref{E:f_rG2}, so that $C \onto D$.  Finally
$$d_{\lambda \up{A}\down{B}, \lambda \up{C} \down {D}}(v) = d_{\lambda \up{A} \down{B},\sigma}(v) \ne 0$$ implies that $(A \cup D) \onto (B \cup C)$ by Lemma \ref{L:dom}, so that $(C,D) \preceq (A,B)$.  Part (iii) now follows from \eqref{E:f_rG}.
\end{enumerate}
\end{proof}

Let $\lambda \in \mathcal{P}_n$ and $r \in \mathbb{Z}_e$.  The induced module $\rInd^{\mathcal{S}_{n+1}} P^{\lambda}$ is projective, so that
$$
\rInd^{\mathcal{S}_{n+1}} P^{\lambda} = \bigoplus_{\tau \in \mathcal{P}_{n+1}} a_{\tau\lambda}^l P^{\tau},
$$
where $a_{\tau\lambda}^l \in \mathbb{N}_0$ satisfies
\begin{align}
a_{\tau\lambda}^l &= \dim_{\mathbb{F}} \Hom_{\mathcal{S}_{n+1}} (\rInd^{\mathcal{S}_{n+1}} P^{\lambda}, L^{\tau}) \notag \\
&= \dim_{\mathbb{F}} \Hom_{\mathcal{S}_{n}} (P^{\lambda}, \rRes_{\mathcal{S}_n} L^{\tau}) \notag \\
&= [\rRes_{\mathcal{S}_n} L^{\tau} : L^{\lambda}] \label{E:a}
\end{align}

We have an analogue of Proposition \ref{P:impt} for $q$-Schur algebras.

\begin{prop} \label{P:imptS}
Let $\lambda \in \mathcal{P}_n$ and $r \in \mathbb{Z}_e$. Let $\rInd P^\lambda = \bigoplus_{\tau \in \mathcal{P}_{n+1}} a_{\tau\lambda}^l P^\tau$.  Suppose that in the Grothendieck group
$$
[\rInd P^{\lambda}] = \sum_{\tau \in \mathcal{P}_{n+1}} c^l_{\tau\lambda} [\Delta^{\tau}].
$$
\begin{enumerate}
\item If $\sigma \in \mathcal{P}_{n+1}$ such that $c^l_{\sigma\lambda} \ne 0$, then $\t_r^{\sigma} \leq \t_r^{\lambda}$.
\item If $\sigma \in \mathcal{P}_{n+1}$ such that $a^l_{\sigma\lambda} \ne 0$, then $\t_r^{\sigma} \leq \t_r^{\lambda}$.
\item If $A \subseteq I_r(\lambda)$ and $B \subseteq R_r(\lambda)$ with $|A| = |B| + 1$, then
$$
c^l_{\lambda \up{A}\down{B}, \lambda} =
\begin{cases}
\sum_{(C,D) \preceq (A,B)} a^l_{\lambda \up{C} \down{D},\lambda} d^l_{\lambda \up {A} \down{B},\lambda \up{C}\down{D}}, &\text{if } A \onto B; \\
0, &\text{otherwise.}
\end{cases}
$$
\end{enumerate}
\end{prop}

The proof of Proposition \ref{P:imptS} is entirely analogous to that of Proposition \ref{P:impt}.

\begin{thm} \label{T:d}
Let $\lambda$ be a partition and let $A \subseteq I_r(\lambda)$, $B \subseteq R_r(\lambda)$.
\begin{enumerate}
\item Then
 $$d_{\lambda \up{A}\down{B}, \lambda}(v) =
 \begin{cases}
 \sum_{\omega \in \Omega(T_r(\lambda)_A^B)} v^{\|\omega\|}, &\text{if } A \leftrightarrow B; \\
 0, &\text{otherwise.}
 \end{cases}
 $$
\item Suppose that $f_r(G(\lambda)) = \sum_{\mu} a_{\mu \lambda}(v) G(\mu)$.  Then $a_{\lambda \up{A} \down{B},\lambda}(v) = 0$ whenever $|A| = |B| + 1$ and $A \onto B$, unless $A = \{a\} \subseteq  V(T_r(\lambda))$ (and $B = \emptyset$), in which case,
    $$
    a_{\lambda\up{\{a\}},\lambda}(v) = [1 + |(U^+(T_r(\lambda)))^{>a}|]_v,
    $$
    where here, and hereafter, for $k \in \mathbb{Z}^+$, we write $[k]_v$ for $v^{-k+1} + v^{-k+3} + \dotsb + v^{k-1}$.
\end{enumerate}
\end{thm}

Theorem \ref{T:d} will be proved in the next section. For the remainder of this section, we shall assume Theorem \ref{T:d} and obtain its analogue for the classical Schur algebras.

Recall that the following branching coefficient is obtained by Kleshchev; note that a node $\mathfrak{n} \in R_r(\lambda)$ is normal if and only if $\mathfrak{n} \in U^+(T_r(\lambda))$:

\begin{thm}[{\cite[Theorem 9.3]{Kleshchev}}] \label{T:Kleshchev}
Let $\lambda \in \mathcal{P}_n$ and $r \in \mathbb{Z}_e$.  Let $b \in U^+(T_r(\lambda))$.  Then
$$
[\rRes_{\mathcal{S}_{n-1}} L^{\lambda} : L^{\lambda \down{\{b\}}}] = 1+|(U^+(T_r(\lambda)))^{>b}|.
$$
\end{thm}

\begin{thm}
Suppose that $e$ is a prime integer.  Let $\lambda \in \mathcal{P}_n$ and $r \in \mathbb{Z}_e$.  Let $A \subseteq I_r(\lambda)$ and $B \subseteq R_r(\lambda)$.
\begin{enumerate}
\item If $|A| = |B|$, then
\begin{align*}
d^e_{\lambda \up{A} \down{B}, \lambda} &= d^0_{\lambda\up{A} \down {B},\lambda} \\
&=
\begin{cases}
|\Omega_A^B(T_r(\lambda))|, &\text{if } A \leftrightarrow B; \\
0, &\text{otherwise.}
\end{cases}
\end{align*}
\item If $|B| = |A|+1$ with $B \onto A$ and either $|B| > 1$ or $B=\{b\} \not\subseteq U^+(T_r(\lambda))$, then for $l = e$, we have
$$
[\rRes_{\mathcal{S}_{n-1}} L^{\lambda} : L^{\lambda \down{B} \up{A}}] = 0.
$$
\end{enumerate}
\end{thm}

\begin{proof}
The Theorem holds for $n = 0$ trivially.  Let $\mu \in \mathcal{P}_{n-1}$, and assume that part (i) holds for $\mu$.  Let $\rInd^{\mathcal{S}_n} P^{\mu} = \bigoplus_{\tau \in \mathcal{P}_{n}} a^l_{\tau\mu} P^{\tau}$, and suppose that in the Grothendieck group, $[\rInd^{\mathcal{S}_n} P^{\mu}] = \sum_{\sigma \in \mathcal{P}_n} c^l_{\sigma\mu} [\Delta^{\sigma}]$. We have, in the Grothendieck group,
\begin{align*}
[\rInd P^{\mu}] &= \sum_{\nu \in \mathcal{P}_{n-1}} d^l_{\nu\mu} [\rInd \Delta^{\nu}]
= \sum_{\t_r^{\nu} \leq \t_r^{\mu}} d^l_{\nu\mu} [\rInd \Delta^{\nu}] \\
&= \sum_{X,Y} d^l_{\mu\up{X}\down{Y},\mu} [\rInd \Delta^{\mu \up{X}\down{Y}}] + \sum_{\t_r^{\nu} < \t_r^{\mu}} d^l_{\nu\mu} [\rInd \Delta^{\nu}],
\end{align*}
where $X$ and $Y$ run over all subsets of $I_r(\mu)$ and $R_r(\mu)$ respectively satisfying $|X| = |Y|$ and $X \onto Y$, by Lemmas \ref{L:decompJorder}, \ref{L:Jorder->preorder} and \ref{L:dom}.  Thus, for $X_1 \subseteq I_r(\mu)$ and $Y_1 \subseteq R_r(\mu)$ such that $|X_1| = |Y_1| + 1$ and $X_1 \onto Y_1$, we have
$$
c^l_{\mu\up{X_1} \down{Y_1}, \mu} = \sum_{z \in R_r(\mu) \cup X_1 \setminus Y_1} d^l_{\mu\up{X_1} \down{Y_1 \cup \{z\}}, \mu}
$$
by Theorem \ref{T:Weylbranch}.
Since part (i) holds for $\mu$, we have
\begin{align*}
c^e_{\mu\up{X_1} \down{Y_1}, \mu} &= \sum_{z \in R_r(\mu) \cup X_1 \setminus Y_1} d^e_{\mu\up{X_1} \down{Y_1 \cup \{z\}}, \mu} \\
&= \sum_{z \in R_r(\mu) \cup X_1 \setminus Y_1} d^0_{\mu\up{X_1} \down{Y_1 \cup \{z\}}, \mu} = c^0_{\mu\up{X_1} \down{Y_1}, \mu},
\end{align*}
so that by Proposition \ref{P:imptS}(iii)
\begin{multline*}
\sum_{(X_2,Y_2) \preceq (X_1,Y_1)} a^e_{\mu \up{X_2}\down{Y_2}, \mu} d^e_{\mu \up{X_1} \down{Y_1}, \mu \up{X_2} \down{Y_2}} \\
= \sum_{(X_2,Y_2) \preceq (X_1,Y_1)} a^0_{\mu \up{X_2}\down{Y_2}, \mu} d^0_{\mu \up{X_1} \down{Y_1}, \mu \up{X_2} \down{Y_2}}.
\end{multline*}
Now $a^0_{\mu \up{X_2}\down{Y_2}, \mu} = a_{\mu \up{X_2}\down{Y_2}, \mu}(1)$ (here, and hereafter, $a_{\tau \mu}(v) \in \mathbb{C}[v,v^{-1}]$ satisfies $f_r(G(\mu)) = \sum_{\tau \in \mathcal{P}_{n+1}} a_{\tau \mu}(v) G(\tau)$), so that $a^0_{\mu \up{X_2} \down{Y_2}} = 0$ if $|X_2| > 1$ or $X_2 = \{x\} \not\subseteq V(T_r(\mu))$ by Theorem \ref{T:d}(ii).  Furthermore, by Theorem \ref{T:Kleshchev} and \eqref{E:a}, $a^e_{\mu \up{X_2} \down{Y_2}, \mu} = a_{\mu \up{X_2}\down{Y_2}, \mu}(1)$ if $X_2 = \{x\} \subseteq V(T_r(\mu))$.  Thus, $a^e_{\mu \up{X_2}\down{Y_2}, \mu} \geq a^0_{\mu \up{X_2}\down{Y_2}, \mu}$ always holds.  We also have $d^e_{\mu \up{X_1} \down{Y_1}, \mu \up{X_2}\down{Y_2}} \geq d^0_{\mu \up{X_1} \down{Y_1}, \mu \up{X_2}\down{Y_2}}$ by Lemma \ref{L:geq}.  Thus
\begin{align*}
\sum_{(X_2,Y_2) \preceq (X_1,Y_1)} a^0_{\mu \up{X_2}\down{Y_2}, \mu}& d^0_{\mu \up{X_1} \down{Y_1}, \mu \up{X_2}\down{Y_2}} \\
&\leq \sum_{(X_2,Y_2) \preceq (X_1,Y_1)} a^e_{\mu \up{X_2}\down{Y_2}, \mu} d^e_{\mu \up{X_1} \down{Y_1}, \mu \up{X_2}\down{Y_2}} \\
&= \sum_{(X_2,Y_2) \preceq (X_1,Y_1)} a^0_{\mu \up{X_2}\down{Y_2}, \mu} d^0_{\mu \up{X_1} \down{Y_1}, \mu \up{X_2}\down{Y_2}},
\end{align*}
forcing equality throughout; i.e.\
\begin{align*}
a^e_{\mu \up{X_2}\down{Y_2}, \mu} &= a^0_{\mu \up{X_2}\down{Y_2}, \mu} ,\\
d^e_{\mu \up{X_1} \down{Y_1}, \mu \up{X_2}\down{Y_2}} &= d^0_{\mu \up{X_1} \down{Y_1}, \mu \up{X_2}\down{Y_2}}
\end{align*}
for all $X_2$ and $Y_2$ such that $(X_2,Y_2) \preceq (X_1,Y_1)$.

Now let $\lambda \in \mathcal{P}_n$ and assume that part (i) holds for all $\mu \in \mathcal{P}_{n-1}$. Let $A \subseteq I_r(\lambda)$ and $B \subseteq R_r(\lambda)$.

If $A = B = \emptyset$, then $d^l_{\lambda \up{A}\down{B}, \lambda} = 1$ for any $l$ by Theorem \ref{T:vdecomp}(i).  If $|A| = |B|>0 $ and $A \not\leftrightarrow B$, then $d^l_{\lambda\up{A}\down{B},\lambda} = 0$ for any $l$ by Lemma \ref{L:dom}.  If $|A| = |B| > 0$ and $A \leftrightarrow B$, then let $b = \max(B)$ and $\mu = \lambda \down{\{b\}} \in \mathcal{P}_{n-1}$.  Let $X_1 = A$, $Y_1 = B \setminus \{b\}$, $X_2 = \{b\}$, $Y_2 = \emptyset$.  Then $X_1,X_2 \subseteq I_r(\mu)$, $Y_1,Y_2 \subseteq R_r(\mu)$ with $(X_2,Y_2) \preceq (X_1,Y_1)$, so that $d^e_{\mu \up{X_1} \down{Y_1}, \mu \up{X_2}\down{Y_2}} = d^0_{\mu \up{X_1} \down{Y_1}, \mu \up{X_2}\down{Y_2}}$, i.e.\ $d^e_{\lambda \up{A}\down{B}, \lambda} = d^0_{\lambda \up{A}\down{B}, \lambda}$.  This proves part (i).

If $|B| = |A| + 1$ with $B \onto A$, and either $|B| > 1$ or $B = \{b\} \not\subseteq U^+(T(\lambda))$, let $\mu = \lambda \down{B} \up {A} \in \mathcal{P}_{n-1}$.  Let $X_1 = B$ and $Y_1 = A$.  Then $X_1 \subseteq I_r(\mu)$, $Y_1 \subseteq R_r(\mu)$ and $X_1 \onto Y_1$, so that
$$
a^e_{\lambda,\lambda \down{B}\up{A}} = a^e_{\mu \up{X_1} \down{Y_1},\mu} = a^0_{\mu \up{X_1} \down{Y_1},\mu} = a_{\mu \up{X_1} \down{Y_1},\mu}(1),
$$
which equals zero by Theorem \ref{T:d}(ii).  Part (ii) now follows from \eqref{E:a}.
\end{proof}

\section{Proof of Theorem \ref{T:d}} \label{S:proof}

We provide the proof of Theorem \ref{T:d} in this section.

Let $\lambda \in \mathcal{P}_n$ and $r \in \mathbb{Z}_e$.  Write $T = (T^+,T^-)$ for $T_r(\lambda) = (R_r(\lambda),I_r(\lambda))$, and let $A \subseteq T^-$ and $B \subseteq T^+$ such that $|A| = |B|+1$ and $A \onto B$.  Consider $\langle f_r( G(\lambda)), \lambda \up{A}\down{B} \rangle$, which can be computed in the following two ways:

Firstly,
\begin{align*}
\langle f_r (G(\lambda)), \lambda \up{A}\down{B} \rangle &= \langle f_r (\sum_{\mu} d_{\mu \lambda}(v) \mu), \lambda \up{A} \down{B} \rangle \\
&= \sum_{c \in T^+ \cup A \setminus B} d_{\lambda \up{A} \down{B\cup \{c\}}, \lambda}(v) \langle f_r (\lambda \up{A} \down{B \cup \{c\}}), \lambda \up{A} \down{B} \rangle \\
&= \sum_{\substack{c \in T^+ \cup A \setminus B\\ A \onto (B \cup \{c\})}} v^{2(|B^{>c}|- |A^{>c}|)-|T_c|} d_{\lambda \up{A} \down{B \cup \{c\}}, \lambda}(v)
\end{align*}
by \eqref{E:f_r} and Lemma \ref{L:dom}.
If Theorem \ref{T:d}(i) holds for $\lambda$, then we have
\begin{equation}
\langle f_r (G(\lambda)), \lambda \up{A}\down{B} \rangle = \sum v^{2(|B^{>c}|- |A^{>c}|) -|T_c| + \| \omega\|} \label{E:left}
\end{equation}
where the sum runs over all $c \in T^+ \cup A \setminus B$ such that $A \onto (B \cup \{c\})$ and over all $\omega \in \Omega(T_A^{B \cup \{c\}})$.

On the other hand, we also have, if $f_r (G(\lambda)) = \sum_{\tau \in \mathcal{P}_{n+1}} a_{\tau\lambda}(v) G(\tau)$,
\begin{equation*}
  \langle f_r (G(\lambda)), \lambda \up{A} \down{B} \rangle = \sum_{(C,D) \preceq (A,B)} a_{\lambda \up{C} \down{D},\lambda}(v) d_{\lambda \up{A} \down{B}, \lambda \up{C}\down{D}}(v)
\end{equation*}
by Proposition \ref{P:impt}(iii).
If Theorem \ref{T:d} holds, then we have
\begin{align}
\langle f_r (G(\lambda)), \lambda \up{A} \down{B} \rangle &=\sum_{d \in V(T)} [1+|(U^+(T))^{>d}|]_v\, d_{\lambda \up{A} \down{B},\lambda\up{\{d\}}}(q) \notag \\
= &\sum_{d \in V(T)} \left(\sum_{d' \in \{d\} \cup (U^+(T))^{>d}} v^{2|T_d^{d'+1}| - |T_d|} \right) d_{\lambda \up{A} \down{B},\lambda\up{\{d\}}}(v) \notag \\
= &\sum v^{2|T_d^{d'+1}| - |T_d| + \| \varpi \|} \label{E:right}
\end{align}
where the last sum runs over all $d \in V(T)$ such that $A \onto (B \cup \{d\})$, $d' \in \{d\} \cup (U^+(T))^{>d}$, and $\varpi \in \Omega((T \up{d})_A^{B\cup \{d\}})$ (where $T\up{d} = T_r(\lambda \up{\{d\}}) = (T^+ \cup \{d \}, T^- \setminus \{ d \})$).

Hence, when Theorem \ref{T:d} holds, there must be a bijection $\phi$ between the indexing sets of the monic monomials in \eqref{E:left} and \eqref{E:right} preserving their respective exponents.

More formally, let
\begin{align*}
\mathcal{L}_A^B(T) &= \{ (c, \omega) \mid c \in  T^+ \cup A \setminus B,\ A \onto (B \cup \{c\}),\ \omega \in \Omega(T_A^{B \cup \{c\}}) \}; \\
\mathcal{R}_A^B(T) &= \{(d,d', \varpi) \mid d \in V(T),\ A \onto (B \cup \{d\}),\\
&\hspace*{2.53cm} d' \in \{d\} \cup (U^+(T))^{>d},\ \varpi \in \Omega((T \up{d})_A^{B\cup \{d\}}) \}.
\end{align*}
For $(c, \omega) \in \mathcal{L}_A^B(T)$ and $(d,d', \varpi) \in \mathcal{R}_A^B(T)$, we define their norms by
\begin{align*}
\| (c, \omega) \| &= 2(|B^{>c}| - |A^{>c}|)-|T_c| + \| \omega \|;  \\
\| (d,d', \varpi) \| &= 2|T_d^{d'+1}| - |T_d| + \| \varpi \|.
\end{align*}
Then we have:

\begin{thm} \label{T:bij}
Let $T = (T^+,T^-)$ be a sign sequence. Let $A \subseteq T^-$ and $B \subseteq T^+$ such that $|A| = |B|+1$ and $A \onto B$.  Then there is a norm-preserving bijection between $\mathcal{L}_A^B(T)$ and $\mathcal{R}_A^B(T)$.
\end{thm}

\begin{thm} \label{equiv}
Theorems \ref{T:d} and \ref{T:bij} are equivalent.
\end{thm}

\begin{proof}
It suffices to show that Theorem \ref{T:bij} implies Theorem \ref{T:d}.  We prove by induction.  We assume Theorem \ref{T:d}(i) holds for all $\lambda \in \mathcal{P}_m$ with $m \leq n$.  Let $\lambda \in \mathcal{P}_n$, and let $A \subseteq I_r(\lambda)$, $B\subseteq R_r(\lambda)$ with $|A| = |B|+1 >0$ and $A \onto B$.
We assume further that Theorem \ref{T:d}(ii) holds for $\lambda$ and all $C \subseteq I_r(\lambda)$, $D\subseteq R_r(\lambda)$ such that $(C,D) \prec (A,B)$.

By Theorem \ref{T:bij}, the right-hand sides of \eqref{E:left} and \eqref{E:right} are equal.  Both are equal to $\langle f_r (G(\lambda)), \lambda \up{A} \down{B} \rangle$ assuming the conjectural formulae in Theorem \ref{T:d} hold.  By our induction hypothesis, these formulae do hold for all terms involved except for $a_{\lambda \up{A} \down{B},\lambda}(v)$ and $d_{\lambda \up{A} \down{B}, \lambda \up{\{x\}}}(v)$, where $x$ indexes the indent $r$-node on the first row of $\lambda$ (if it exists; the formula holds for other $d_{\lambda \up{A} \down{B}, \lambda \up{\{d\}}}(v)$ by Theorem \ref{T:vdecomp}(v)).  We are thus reduced to the following equations:
\begin{alignat*}{2}
a_{\lambda \up{\{x\}},\lambda}(v) d_{\lambda \up{\{x\}}, \lambda\up{\{x\}}}(v) &= \mathfrak{a}_{\lambda \up{\{x\}} ,\lambda}(v) \mathfrak{d}_{\lambda \up{\{x\}}, \lambda\up{\{x\}}}(v) &\quad &((A,B) = (\{x\}, \emptyset)) \\
a_{\lambda \up{A} \down{B},\lambda}(v) + d_{\lambda \up{A} \down{B}, \lambda\up{\{x\}}}(v) &= \mathfrak{a}_{\lambda \up{A} \down{B},\lambda}(v) + \mathfrak{d}_{\lambda \up{A} \down{B}, \lambda\up{\{x\}}}(v) && ((A,B) \ne (\{x\}, \emptyset)),
\end{alignat*}
where $\mathfrak{a}_{\lambda \up{A} \down{B},\lambda}(v)$ and $\mathfrak{d}_{\lambda \up{A} \down{B}, \lambda\up{\{x\}}}(v)$ denote the conjectural formulae for $a_{\lambda \up{A} \down{B},\lambda}(v)$ and $d_{\lambda \up{A} \down{B}, \lambda\up{\{x\}}}(v)$ asserted by Theorem \ref{T:d} respectively.  By Theorem \ref{T:vdecomp}(i), we have $d_{\lambda \up{\{x\}}, \lambda\up{\{x\}}}(v) = 1 = \mathfrak{d}_{\lambda \up{\{x\}}, \lambda\up{\{x\}}}(v)$, and hence $a_{\lambda \up{\{x\}},\lambda}(v) = \mathfrak{a}_{\lambda \up{\{x\}},\lambda}(v)$ from the first equation.  On the other hand, since $a_{\lambda \up{A} \down{B},\lambda}(v)$ and $\mathfrak{a}_{\lambda \up{A} \down{B},\lambda}(v)$ are Laurent polynomials in $v$ symmetric about $v^0$, while $d_{\lambda \up{A} \down{B}, \lambda\up{\{x\}}}(v)$ and $\mathfrak{d}_{\lambda \up{A} \down{B}, \lambda\up{\{x\}}}(v)$ are polynomials in $v$ with no $v^0$ term, the second equation yields
$$
a_{\lambda \up{A} \down{B},\lambda}(v)  = \mathfrak{a}_{\lambda \up{A} \down{B},\lambda}(v) \qquad \text{and} \qquad d_{\lambda \up{A} \down{B}, \lambda\up{\{x\}}}(v) = \mathfrak{d}_{\lambda \up{A} \down{B}, \lambda\up{\{x\}}}(v).
$$
Thus Theorem \ref{T:d}(ii) holds for $(A,B)$, and hence for all partitions of $n$ by induction.

Let $\mu \in \mathcal{P}_{n+1}$, and let $X \subseteq I_r(\mu)$, $Y \subseteq R_r(\mu)$.  If $\mu$ does not have an indent $r$-node on its first row, or if $\mu$ has an indent $r$-node on its first row, indexed by $x$ say, with $x \notin Y$, then $d_{\mu \up{X} \down{Y}, \mu}(v)$ equals the conjectural formula of Theorem \ref{T:d}(i) by Theorem \ref{T:vdecomp}(v).  If $\mu$ has an indent $r$-node on its first row, indexed by $x$, and $x \in Y$,  then let $\lambda = \mu \down{\{x\}}$.  We have seen from above that $d_{\mu \up{X} \down{Y}, \mu}(v) = d_{\lambda \up{X}\down{Y \setminus \{x\}},\lambda \up{\{x\}}}(v)$ equals the conjectural formula of Theorem \ref{T:d}(i).  Thus Theorem \ref{T:d}(i) holds for $\mu$, and the proof is complete.
\end{proof}

We end the paper by providing a proof of Theorem \ref{T:bij}, thereby completing the proof for Theorem \ref{T:d}.

\begin{proof}[Proof of Theorem \ref{T:bij}]
We prove by induction, and consider the following three cases separately.
\begin{itemize}
\item $B = \emptyset$:

This is the base case for the induction, and is essentially proved in \cite{CMT}.  More specifically, we define $\phi : \mathcal{L}_{\{a\}}^{\emptyset}(T) \to \mathcal{R}_{\{a\}}^{\emptyset}(T)$ as follows.  Let $(c, \gamma) \in \mathcal{L}_{\{a\}}^{\emptyset}(T)$.

For $c\in P^+(T)$, let $d = \min(V(T)^{>c})$, and let $\widehat{\gamma}$ be the latticed path for $(T\up{d})_a^d$ obtained by extending the latticed path $\gamma$ for $T_a^c$ by the $\diagup$ at $c$ and then followed by the descending latticed path between $c$ and $d$.  Then $\phi(c, \gamma) = (d,d,\widehat{\gamma})$.
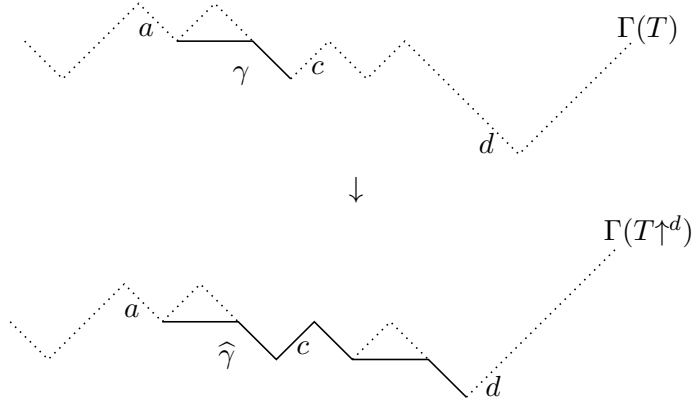
\medskip

\begin{center}
  \begin{tikzpicture}[scale= 0.5]
    \draw [semithick,dotted] (-4,0) -- (-3,-1) -- (-1,1) -- (0,0) -- (1,1) -- (2,0);
    \draw [semithick,dotted] (3,-1) -- (4,0) -- (5,-1) -- (6,0) -- (7,-1) -- (8,-2) -- (9,-3) -- (12,0);
    \draw [semithick] (0,0) -- (2,0) -- (3,-1);
    \draw (-0.8, 0.3) node {$a$};
    \draw (3.7, -0.7) node {$c$};
    \draw (8.15, -2.7) node {$d$};
    \draw (12.4, 0.3) node {$\Gamma(T)$};
    \draw (1.7, -0.9) node {$\gamma$};
  \end{tikzpicture} \\
  $\downarrow$ \\
  \begin{tikzpicture}[scale= 0.5]
    \draw [semithick,dotted] (-4,0) -- (-3,-1) -- (-1,1) -- (0,0) -- (1,1) -- (2,0);
    \draw [semithick,dotted] (5,-1) -- (6,0) -- (7,-1);
    \draw [semithick,dotted] (8,-2) -- (12,2);
    \draw [semithick] (0,0) -- (2,0) -- (3,-1) -- (4,0) -- (5,-1) -- (7,-1) -- (8,-2);
    \draw (-0.8, 0.3) node {$a$};
    \draw (3.7, -0.7) node {$c$};
    \draw (8.7, -1.7)node {$d$};
    \draw (12.8, 2.4) node {$\Gamma(T\up{d})$};
    \draw (1.7, -0.9) node {$\widehat{\gamma}$};
  \end{tikzpicture}
\end{center}

For $c\in U^+(T)$, let $d = \max((T^-)^{<c})$ ($=\max(V(T)^{<c})$), and let $\bar{\gamma}$ be the latticed path for $(T\up{d})_a^d$ obtained by truncating the latticed path $\gamma$ for $T_a^c$ at $d$.  Then $\phi(c, \gamma) = (d,c,\bar{\gamma})$.

For $c = a$ (and hence $\gamma = \emptyset$) with $a \in V(T)$, $\phi(a,\emptyset) = (a,a,\emptyset)$.  For $c=a$ with $a \notin V(T)$, let $d = \min(V(T)^{>c})$, and let $\delta$ be the descending latticed path for $(T\up{d})_a^d$.  Then $\phi(a,\emptyset) = (d,d,\delta)$.

The reader may check that $\phi$ so defined is a well-defined norm-preserving bijection.\medskip

\item $|B| > 0$, and $(T^+)_{a_0}^{b_0} = \emptyset$ for some $a_0 \in A$, $b_0 \in B$ with $a_0< b_0$:

By replacing $a_0$ with $\max((T^-)_{a_0}^{b_0} \cap A)$ if necessary, we may assume that $(T^-)_{a_0}^{b_0} \cap A = \emptyset$.  Then $\pi_{A \to B \cup\{c\}} (a_0) = (b_0)$ for all $c \ne S^+(T) \cup A \setminus (B \cup \{a_0\})$, and $\pi_{A \to B \cup\{d\}} (a_0) = (b_0)$ for all $d \in V(T) \setminus \{a_0\}$, and the generic latticed path is the only element in $\mathbb{L}(T_{a_0}^{b_0})$.
Let $A' = A \setminus \{a_0\}$, $B' = B \setminus \{b_0\}$.  We define $\phi : \mathcal{L}_A^B(T) \to \mathcal{L}_{A'}^{B'}(T)$ and $\psi : \mathcal{R}_A^B(T) \to \mathcal{R}_{A'}^{B'}(T)$ as follows.

For $(c,\omega) \in \mathcal{L}_A^B(T)$, with $\omega = (\gamma_a)_{a \in A}$, let $\omega' = (\gamma_a)_{a \in A'}$.  Define
$$
\phi(c,\omega) :=
\begin{cases}
(c, \omega'), &\text{if $c \ne a_0$;} \\
(b_0, \omega'), &\text{if $c = a_0$}.
\end{cases}
$$

For $(d,d', \varpi) \in \mathcal{R}_A^B(T)$, with $\varpi = (\delta_a)_{a\in A}$, define $\psi(d,d', \varpi) :=
(d,d', \varpi')$, where $\varpi' = (\delta_a)_{a\in A'}$.

The reader may check that both $\phi$ and $\psi$ are well-defined bijections that reduce the norm of each element in their respective domains by $1-|T_{a_0}^{b_0}|$. By induction hypothesis (on $|B|$), there is a norm-preserving bijection $\chi: \mathcal{L}_{A'}^{B'}(T) \to \mathcal{R}_{A'}^{B'}(T)$.  Thus $\psi^{-1} \circ \chi \circ \phi : \mathcal{L}_A^B(T) \to \mathcal{R}_A^B(T)$ is a norm-preserving bijection. \medskip

\item $|B| > 0$, and $(T^+)_a^b \ne \emptyset$ for all $a \in A$, $b \in B$ with $a< b$:

This is the most difficult case.  Let $a_0 \in A$, $b_0 \in B$ with $a_0< b_0$, such that $|(T^+)_{a_0}^{b_0}| \leq |(T^+)_a^b|$ for all $a \in A$, $b \in B$ with $a < b$.  Then $(T^{\pm})_{a_0}^{b_0} \cap B = \emptyset$ and, by replacing $a_0$ with $\min((T^{\pm})_{a_0}^{b_0} \cap A)$ if necessary, we may assume that $(T^{\pm})_{a_0}^{b_0} \cap A = \emptyset$. Let $b_1 = \max((T^+)_{a_0}^{b_0})$, and let $\tilde{B} = B \cup \{b_1\} \setminus \{b_0\}$.  If $T_{b_1}^{b_0} \ne \emptyset$, then $(T^{\pm})_{b_1}^{b_0} = (T^-)_{b_1}^{b_0}$; let $a_2 = \min((T^{\pm})_{b_1}^{b_0}))$, and $T' = (T^+ \setminus \{b_1\}, T^- \setminus \{ a_2\})$.

Let $\phi_1: \mathcal{L}_A^{\tilde{B}}(T) \to \mathcal{L}_A^B(T) $ be defined as follows.  Let $(c,\tilde{\omega}) \in \mathcal{L}_A^{\tilde{B}}(T)$.  For $c \ne b_0$, let $a_1 = \pi^{-1}_{A \to (\tilde{B} \cup \{c\})}(b_1)$. Then $\pi_{A \to (B \cup \{c\})}(a_1) = b_0$.  Let $\omega \in \Omega(T_A^{B \cup \{c\}})$ be obtained from $\tilde{\omega} = (\tilde{\gamma}_a)_{a\in A}$ by replacing $\tilde{\gamma}_{a_1}$ with the latticed path for $T_{a_1}^{b_0}$ obtained by extending $\tilde{\gamma}_{a_1}$ by the generic latticed path between $b_1$ (inclusive) and $b_0$.  Define
$$
\phi_1(c,\tilde{\omega}) =
\begin{cases}
(c,\omega), &\text{if }c \ne b_0; \\
(b_1,\tilde{\omega}), &\text{if }c = b_0.
\end{cases}
$$

Let $\phi_2: \mathcal{L}_A^B(T') \to \mathcal{L}_A^B(T)$ be defined as follows.  Let $(c,\omega') \in \mathcal{L}_A^B(T')$, with $\omega' = (\gamma'_a)_{a \in A}$.  Write $\pi$ for $\pi_{A \to (B \cup \{c\})}$.  For $a \in A$, define $\gamma_a \in \mathbb{L}(T_a^{\pi(a)})$ to be the latticed path obtained by inserting horizontal line segments into $\gamma'_a$ at positions $b_1$ and $a_2$ if $a < b_1 < \pi(a)$, and as $\gamma'_a$ otherwise.  Let $\omega = (\gamma_a)_{a \in A}$.  Then $\omega \in \Omega(T_A^{B \cup \{c\}})$.  Define $\phi_2(c,\omega') = (c, \omega)$.

The reader may check that $\phi_1$ and $\phi_2$ are injective, whose images partition $\mathcal{L}_A^B(T)$.  Furthermore, $\phi_1$ increases the norm of each element by $1- |T_{b_1}^{b_0}|$, while $\phi_2$ is norm-preserving.

Let $\psi_1: \mathcal{R}_A^{\tilde{B}}(T) \to \mathcal{R}_A^B(T) $ be defined as follows.  Let $(d,d', \tilde{\varpi}) \in \mathcal{R}_A^{\tilde{B}}(T)$, with $\tilde{\varpi} = (\tilde{\delta}_a)_{a\in A}$.
For $d \ne \max((T^{\pm})^{b_0})$, let $a_1 = \pi^{-1}_{A \to (\tilde{B} \cup \{d\})}(b_1)$.  Then $\pi_{A \to (B \cup \{d\})} (a_1) = b_0$.  Let $\varpi_1 \in \Omega((T\up{d})_A^{B \cup \{d\}})$ be obtained from $\tilde{\varpi}$ by replacing $\tilde{\delta}_{a_1}$ with the latticed path for $(T\up{d})_{a_1}^{b_0}$ obtained by extending $\tilde{\delta}_{a_1}$ by the generic latticed path between $b_1$ (inclusive) and $b_0$.
For $d = \max((T^{\pm})^{b_0})$, let $a' = \pi^{-1}_{A \to (\tilde{B} \cup \{d\})}(d)$.  Then
\begin{alignat*}{2}
\pi_{A \to (\tilde{B} \cup \{d\})} (a_0) &= b_1, \qquad & \pi_{A \to (\tilde{B} \cup \{d\})} (a') &= d, \\
\pi_{A \to (B \cup \{d\})} (a_0) &= d, \qquad & \pi_{A \to (B \cup \{d\})} (a') &= b_0.
\end{alignat*}
Let $\varpi_2 \in \Omega((T\up{d})_A^{B \cup \{d\}})$ be obtained from $\tilde{\varpi}$ by replacing $\tilde{\delta}_{a_0}$ with the latticed path for $(T\up{d})_{a_0}^{d}$ obtained by extending $\tilde{\delta}_{a_0}$ by the generic latticed path between $b_1$ (inclusive) and $d$, and replacing $\tilde{\delta}_{a'}$ with the latticed path for $(T\up{d})_{a'}^{b_0}$ obtained by extending $\tilde{\delta}_{a'}$ by the generic latticed path between $d$ (inclusive) and $b_0$ (which is just $\diagup$ at $d$).  Define
$$\psi_1(d,d',\tilde{\varpi}) =
\begin{cases}
(d,d',\varpi_1) &\text{if }d \ne \max((T^{\pm})^{b_0}); \\
(d,d',\varpi_2) &\text{if }d = \max((T^{\pm})^{b_0}).
\end{cases}
$$

Let $\psi_2: \mathcal{R}_A^B(T') \to \mathcal{R}_A^B(T) $ be defined as follows.  Let $(d,d', \varpi') \in \mathcal{R}_A^B(T')$, with $\varpi' = (\delta'_a)_{a \in A}$.  Write $\pi$ for $\pi_{A \to (B \cup \{d\})}$.  For $a \in A$, define $\delta_a \in \mathbb{L}((T\up{d})_a^{\pi(a)})$ to be the latticed path obtained by inserting horizontal line segments into $\delta'_a$ at positions $b_1$ and $a_2$ if $a < b_1 < \pi(a)$, and as $\delta'_a$ otherwise.  Let $\varpi = (\delta_a)_{a \in A}$.  Then $\varpi \in \Omega((T\up{d})_A^{B \cup \{c\}})$.  Define $\psi_2(d,d', \varpi') = (d,d', \varpi)$.

The reader may check that $\psi_1$ and $\psi_2$ are injective, whose images partition $\mathcal{R}_A^B(T)$.  Furthermore, $\psi_1$ increases the norm of each element by $1- |T_{b_1}^{b_0}|$, while $\psi_2$ is norm-preserving.

By induction hypothesis (on $\min \{ |(T^+)_a^b| : a \in A,\ b \in B,\ a<b\}$), there are norm-preserving bijections $\varphi: \mathcal{L}_A^{\tilde{B}}(T) \to \mathcal{R}_A^{\tilde{B}}(T)$ and $\chi : \mathcal{L}_A^B(T') \to \mathcal{R}_A^B(T')$.  Thus, $(\psi_1 \coprod \psi_2) \circ (\varphi \coprod \chi) \circ (\phi_1 \coprod \phi_2)^{-1}$ is a norm-preserving bijection from $\mathcal{L}_A^B(T)$ to $\mathcal{R}_A^B(T)$.
\end{itemize}
\end{proof}


\begin{thebibliography}{99}


\bibitem{removal}
J. Chuang, H. Miyachi, K. M. Tan, `Row and column removal in the $q$-deformed Fock space', {\em J. Algebra} \textbf{254} (84--91), 2002.

\bibitem{CMT}
J. Chuang, H. Miyachi, K. M. Tan, `Kleshchev's decomposition numbers and branching coefficients in the Fock space', {\em Trans. Amer. Math. Soc.} \textbf{360} (1179--1191), 2008.

\bibitem{DJ}
R. Dipper, G. D. James, `The $q$-Schur algebra', {\em Proc. London Math. Soc. (3)} \textbf{59} (23--50), 1989.

\bibitem{J}
G. D. James, `The decomposition matrices for $\mathrm{GL}_n(q)$ for $n \leq 10$', {\em Proc. London Math. Soc. (3)} \textbf{60} (225--265), 1990.

\bibitem{Mathasbook}
A. Mathas, {\em Iwahori-Hecke algebras and Schur algebras of the symmetric group}, Providence, RI, American Mathematical Society, 1999.

\bibitem{Kleshchev}
A. Kleshchev, `On decomposition numbers and branching coefficients for symmetric and special linear groups', {\em Proc. London. Math. Soc. (3)} \textbf{75} (497--558), 1997.

\bibitem{L}
B. Leclerc, `Symmetric functions and the Fock space', {\em Symmetric functions 2001: surveys of developments and perspectives} (153--177), NATO Sci. II Math. Phys. Chem. 74, Kluwer Acad. Publ., Dordrecht, 2002.

\bibitem{LT1}
B. Leclerc,  J.-Y. Thibon, `Canonical bases of $q$-deformed Fock spaces', {\em Internat. Math. Res. Notices} \textbf{1996} (447--456), 1996.

\bibitem{VV}
M. Varagnolo, E. Vasserot, `On the decomposition matricecs of the quantized Schur algebra', {\em Duke Math. J.} \textbf{100} (267--297), 1999.


\end{thebibliography}
\end{document}